\documentclass[11pt]{amsart}
\usepackage{epsfig}
\usepackage{amsmath}
\usepackage{amsthm, amssymb, amscd}
\DeclareSymbolFont{AMSb}{U}{msb}{m}{n}
\DeclareMathSymbol{\N}{\mathbin}{AMSb}{"4E}
\DeclareMathSymbol{\Z}{\mathbin}{AMSb}{"5A}
\DeclareMathSymbol{\R}{\mathbin}{AMSb}{"52}
\DeclareMathSymbol{\Q}{\mathbin}{AMSb}{"51}
\DeclareMathSymbol{\I}{\mathbin}{AMSb}{"49}
\DeclareMathSymbol{\C}{\mathbin}{AMSb}{"43}

\begin{document}

\title{Forming the Borromean Rings out of arbitrary polygonal unknots}
  \date{\today}
\author[H.N.\ Howards]{Hugh Nelson Howards}
    
\address{Department of Mathematics, Wake Forest University, Winston-Salem, NC 27109, USA}
\maketitle

\newtheorem{definition}{Definition}[section]
\newtheorem{corollary}[definition]{Corollary}
\newtheorem{conjecture}[definition]{Conjecture}
\newtheorem{cor}[definition]{Corollary}
\newtheorem{theorem}[definition]{Theorem}
\newtheorem{lemma}[definition]{Lemma}
\newtheorem{claim}[definition]{Claim}
\newtheorem{ex}[definition]{Example}
\newtheorem{q}[definition]{Question}
\newtheorem{exer}[definition]{Exercise}
\newtheorem{proposition}[definition]{Proposition}

\newcommand{\bi}{\begin{itemize}}
\newcommand{\ei}{\end{itemize}}
\newcommand{\be}{\begin{enumerate}}
\newcommand{\ee}{\end{enumerate}}
\newcommand{\ds}{\displaystyle}
\newcommand{\hs}{Heegaard splitting}
\newcommand{\tp}{thin position}
\newcommand{\sud}{strict upper disk}
\newcommand{\sld}{strict lower disk}
\newcommand{\db}{disk busting}
\newcommand{\si}{strongly irreducible}
\newcommand{\dbc}{double branched cover}
\newcommand{\ul}{\underline}
\newcommand{\po}{\Pi_{\omega}}

\begin{abstract}
We prove the perhaps surprising result that given any three polygonal unknots in $\R^3$, then we may form the Borromean rings out of them through rigid motions of $\R^3$ applied to the individual components together with possible scaling of the components.  We also prove that if at least two of the unknots are planar, then we do not need scaling.  This is true even for a set of three polygonal unknots that are arbitrarily close to three circles, which themselves cannot be used to form the Borromean Rings.
\end{abstract}

\section{Introduction}
\begin{figure}
	\centering
	\includegraphics[width=0.35\textwidth]{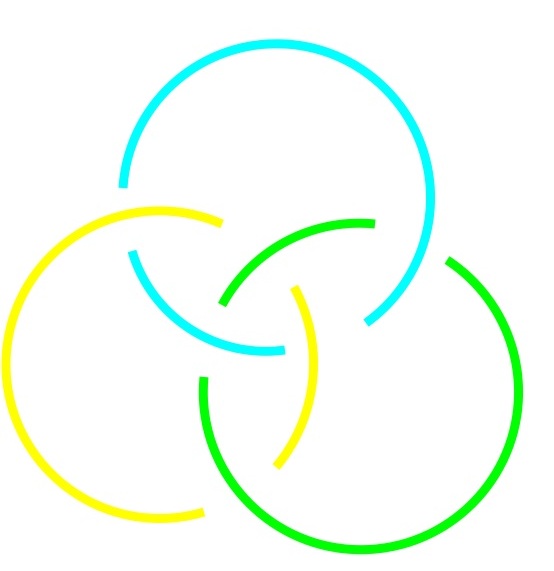}
	\caption{The Borromean Rings.}
	\label{fig:br}

	\end{figure}

The Borromean Rings in Figure~\ref{fig:br} appear to be made out of circles, but a result of Freedman and Skora shows that this is an optical illusion (see \cite{F} or \cite{H1}).
The Borromean Rings are a special type of Brunninan Link: a link of $n$ components is one which is not an unlink, but for which every sublink of $n-1$ components is an unlink.  There are an infinite number of distinct Brunnian links of $n$ components for $n \geq 3$, but the Borromean Rings are the most famous example.

This fact that the Borromean Rings cannot be formed from three circles often comes as a surprise, but then we come to the contrasting result that although it cannot be built out of circles, the Borromean Rings can be built out of certain sets of convex curves.  For example, one can form it from two circles and an ellipse.  Although it is only one out of an infinite number of Brunnian links of three components, it is the only one which can be built out of convex components \cite{H1}.  The convexity result is, in fact a bit stronger and shows that no 4 component Brunnian link can be made out of convex components and Davis generalizes this result to 5 components in \cite{D}.

While this shows it is in some sense hard to form most Brunnian links out of certain shapes, the Borromean Rings leave some flexibility.  This leads to the question of what shapes can be used to form the Borromean Rings and the following surprising conjecture of Matthew Cook at the California Institute of Technology.

\begin{conjecture} (Cook)
Given any three unknotted simple closed curves in $\R^3$, they can always be arranged to form the Borromean Rings unless they are all circles. \cite{Cook}

\label{conj:cook}
\end{conjecture}

In this paper we show that any three polygonal unknots (consisting of straight edges meeting at a set of vertices) can be used to form the Borromean Rings through rigid transformations  of the components in $\R^3$ together with scaling of $\R^3$ applied to the individual components.  Note that since any knot can be approximated with a polygonal knot that is arbitrarily close to it, any set of three unknots comes arbitrarily close to forming the Borromean rings: even three circles which themselves cannot form the Borromean Rings. 

While the first few sections prove that many sets of unknots do satisfy Cook's conjecture we conclude in the final section with a possible counter-example to Cook's conjecture.

\section{The main theorems}

To prove our theorems we will need the following lemma about unknots and the disks they bound.  Throughout the paper, $K_i$ will bound a disk $D_i$ and we will abuse notation by preserving the names $K_i$ and $D_i$ even after rigid motions or scaling of the components.   When we refer to a disk or sub-disk as flat or planar, we mean that it is a subset of a flat plane.

For each polygonal unknot $K_i$ 
 we pick can pick an extremal vertex  (a vertex that is a global maximum with respect to some direction vector) which we will call $v_i$.  The edges adjacent to $v_i$ will be called $e_i$ and $f_i$.

\begin{lemma}
If $v_i$ is a unique global maximum for an unknot $K_i$ (with respect to some direction vector) then we may choose a disk $D_i$ whose boundary is $K_i$ and which also has the point $v_i$ as its unique global maximum in the same direction.
\label{lemma:extremaldisks}
\end{lemma}

This lemma will certainly hold for the special case of polygonal unknots that we study in this paper, but we prove it in general.

\begin{proof}
This can be done with a standard innermost loop argument.  Since there is a plane that intersects $K_i$ in $v_i$ and otherwise contains $K_i$ entirely on one side of it, we can also find
 a sphere tangent to the plane at $v_i$, intersecting the knot only in $v_i$ and which otherwise contains $K_i$ entirely inside of it (as the radius of the spheres tangent to the plane at $v_i$ goes to infinity, the spheres limit on the plane).  Now pick an embedded disk $D_i$ for $K_i$ whose interior intersects $S$ transversally in a minimal number of components.  Since $K_i \cap S$ is a single point there are no arcs of intersection in $S \cap D_i$.  This means all remaining intersections may be assumed to be circles.  If the set of circles is nontrivial take an innermost circle on $S$ (one of the components of $S \cap D_i$ that bounds a disk on $S$ disjoint on its interior from $S \cap D_i$) and
cut and paste $D_i$ replacing the component of $D_i -  (D_i \cap S)$ that is bounded by this circle and does not contain $K_i$ by the corresponding subset of $S$.  Pushing the new disk slightly off of $S$ gives a new disk $D_i'$ that intersects $S$ fewer times than $D_i$ did yielding a contradiction to the minimality assumption and showing that we may assume $D_i \cap S = v_i$.

\end{proof}

\begin{claim}
Given polygonal unknots $K_1$, $K_2$, and $K_3$  and disks $D_1$, $D_2$, and $D_3$ as in Lemma~\ref{lemma:extremaldisks} there exists an $\epsilon >0$ such that we may assume that each $D_i$ is planar in an $\epsilon$ neighborhood of $v_i$ (a subset of the plane containing edges $e_i$ and $f_i$), but such that $v_i$ is still the unique global maximum for $D_i$.   
\label{claim:flat}
\end{claim}

\begin{proof} The argument is simple.  We have already shown that there is a plane that intersects $D_i$ only in $v_i$.  We may find a plane parallel to it that intersects $D_i$ only in an arc.  We may need an isotopy of $D_i$ to straighten all such arcs near $v_i$, but $K_i$ stays fixed.

\end{proof}

Translate the three knots so that each $v_i$ is at the origin. Let the closure of the complement of an $\epsilon$ neighborhood of $v_i \subset D_i$ be called $J_i$.  Recall that we have just asserted above that we may assume that $D_i - J_i$ is planar since this is just a small neighborhood of $v_i$.

\begin{lemma}
Given polygonal unknots $K_1$, $K_2$, and $K_3$, disks $D_1$, $D_2$, and $D_3$, and subsets $J_1$, $J_2$, and $J_3$  as above then we may scale $K_3$ and $D_3$ up so that every point in $J_3$ is farther from the origin than the distance of any point in $D_2$ from the origin and we may then scale up $K_1$ and $D_1$ so that every point in $J_1$ is farther from the origin than any point in $D_2$ or $D_3$.
\label{lemma:extremaldisks2}
\end{lemma}

\begin{proof}

For each $i$, let $v_i$ be at the origin and let $r_i$ be the maximum distance from the origin to any point of $D_i$.  
We know by the argument above that we may pick an $\epsilon$ such that
an $\epsilon$ neighborhood of the origin in $\R^3$ intersects each $D_i$ only in a flat subset.  Scale $K_3$ up by multiplying by the three by three matrix $\lambda I$, where $\lambda > \frac{r_2}{\epsilon}$.  As mentioned earlier, we will abuse notation and call the scaled up knot and disk $K_3$ and $D_3$.  The scaling will ensure that any point on $D_3$ that is not in the planar portion is farther from the origin than any point in $D_2$.  
Now scale $D_1$ and $K_1$ up similarly so that any point on $D_1$ that is not in its planar portion is farther from the origin than any point in $D_2$ or $D_3$.

\end{proof}

\begin{cor}
We may scale the knots and disks so that if each $v_i$ is sufficiently close to the origin then $D_2$ can only intersect the planar portion of $D_3 \cup D_1$ and such that $D_3 \cap D_1$ is contained in the planar portion of $D_1$.
\label{cor:planar}
\end{cor}

We now state the two main theorems of the paper.

\begin{theorem}
Let $K_1$, $K_2$, and $K_3$ be three polygonal unknots, then we may form the Borromean rings out of them through rigid motions of $\R^3$ applied to the individual components together with scaling of the components.
\label{thm:poly}
\end{theorem}

\begin{theorem}
Let $K_1$, $K_2$, and $K_3$ be three polygonal unknots, at least two of which are planar, then we may form the Borromean rings out of them through rigid motions of $\R^3$ applied to the individual components.
\label{thm:polyplanar}
\end{theorem}

The scaling of Lemma~\ref{lemma:extremaldisks2} and Corollary~\ref{cor:planar} is the only scaling necessary in the proof of Theorem~\ref{thm:poly} and no scaling is necessary in the proof of Theorem~\ref{thm:polyplanar}.  Not using Lemma~\ref{lemma:extremaldisks2} or Corollary~\ref{cor:planar}  in Theorem~\ref{thm:polyplanar} is the only difference between the two proofs.

%
%
%To preview the argument and build intuition, we now give an outline of how the knots will be positioned in the proof.  We first pick an extremal vertex for each of the three knots and position the knots so that all three extremal vertices are at the origin, but for $K_1$ the pair of edges leaving the extremal vertex are horizontal (lie in the $xy$-plane) and the other two knots have edges leaving the critical vertex that are vertical, lying in $P$, a plane containing the $y$-axis, other than the $xy$-plane (often it will be possible to let $P$ be
%the $yz$-plane) as in Figure~\ref{fig:origin}.  For the two knots with a pair of edges in $P$, we position them so that one has its global maximum at the origin and the other has its global minimum at the origin.  At this point if the knot with two horizontal edges is not planar, we scale it up until all its edges other than the two leaving the critical vertex are far from the other two knots, but keeping the vertex at the origin fixed - the scaling is unnecessary if this knot is planar.
%
%The knots and disks are moved minimally to put the disks in general position.  In the process, we ensure that each $D_r$ will intersect $D_s \cup D_t$, $r \neq s \neq t$, in a pair of crossed arcs, one of which will have its end points on $K_r$ and the other having its end points in the interior of $D_r$ as in Figure~\ref{fig:brstick5}.  This intersection pattern occurring on all three disks is well known to imply the link is the Borromean Rings.
%
%With the outline in mind, we now provide the complete details of the proof.
\bigskip

{\em Proof of Theorems~\ref{thm:poly} and \ref{thm:polyplanar}:} The two proofs are nearly identical, so it will be easy to prove both at the same time.  We start by arguing that we may use rigid transformations of $\R^3$ to position $K_1$, $K_2$, and $K_3$ as they appear in  Figure~\ref{fig:origin} and then use a
translation to arrive at  Figure~\ref{fig:polygonal}.  In the case where two of the components are planar, let them without loss of generality, be $K_1$ and $K_3$.

\begin{figure}
	\centering
	\includegraphics[width=0.45\textwidth]{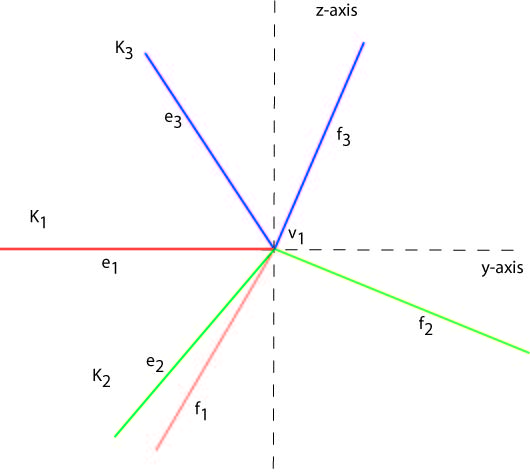}
	\includegraphics[width=0.45\textwidth]{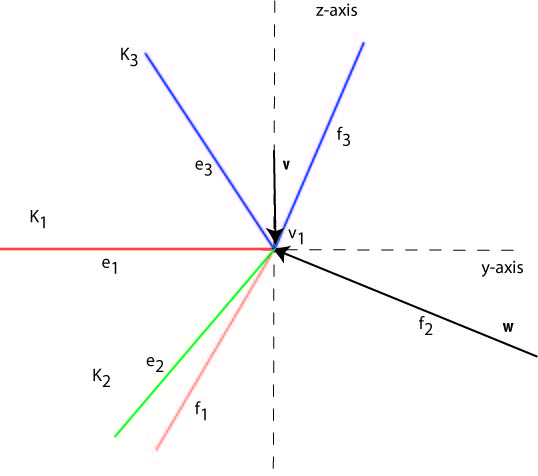}
	\caption{The Knots are initially moved so they have an extremal vertex at the origin.  $\{e_1, f_1\} \subset K_1$ lie in the $xy$-plane,  $\{e_2, f_2\} \subset K_2$ and  $\{e_3, f_3\} \subset K_3$ lie in a plane $P$ containing the $y$ axis.
	 On the right we see a vector ${\vec w}$ with its head at $v_2$ and its tail at the other vertex of $f_2$ and we see the vector
${\vec v}$  with its head at $v_3$ at the origin and its tail between $e_3$ and $f_3$.}
	\label{fig:origin}

	\end{figure}

%
%\begin{figure}
%	\centering
%	\includegraphics[width=0.5\textwidth]{vectors2.jpg}
%	\caption{The vector ${\vec w}$ has its head at $v_2$ and its tail at the other vertex of $f_2$ and the vector
%${\vec v}$  has its head on $v_3$ at the origin and its tail between $e_3$ and $f_3$, although it is vertical here, it need not be in general.}
%	\label{fig:vector}

%	\end{figure}

\begin{figure}
	\centering
	\includegraphics[width=0.65\textwidth]{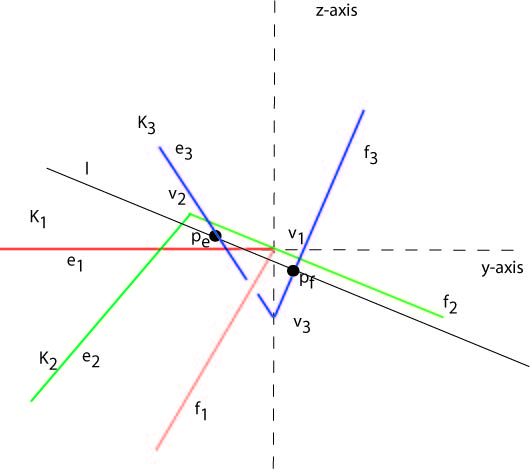}
	\caption{We shift $K_2$ slightly in the direction of the vector ${\vec w}$ and $K_3$ slightly in the direction of ${\vec v}$.}
	\label{fig:polygonal}

	\end{figure}
%\begin{figure}
%	\centering
%	\includegraphics[width=0.65\textwidth]{sticksb.jpg}
%	\caption{Setting up the Borromean Rings. $K_2$ and $K_3$ are currently in the $yz$-plane
%	while $K_1$ is in the $xy$-plane.}
%	\label{fig:polygonal}
%
%	\end{figure}

We initially position $v_1$, $v_2$, and $v_3$  at the origin. To be specific, for $K_1$  place $e_1$
 in the $xy$-plane so that $v_1$ is at the origin and $e_1$ lies on the (negative)  $y$-axis.  Fixing this edge rotate $K_1$ until  $f_1$ lies in the $xy$-plane and has positive values for $x$-coordinates aside from at the point $v_1$ which has $x$ coordinate 0.

We will arrange $K_2$ and $K_3$ so that $e_2$, $f_2$, $e_3$, and $f_3$ all contain the origin, and are all coplanar in a plane $P$ that contains the $y$ axis, such as the $xz$-plane.  This requires a little work since we also care about the intersections of the $D_i$'s.

Place $K_2$ so that $v_2$ is a global maximum for $K_2$ with respect to $z$,  and lies at the origin.  Rotate it around the $z$ axis so that $P_2$, the plane containing $e_2$ and $f_2$ contains the $y$ axis.  $v_2$, of course, remains a global maximum.

Place $K_3$ so that $v_3$ is a global minimum for $K_3$ with respect to $z$,  and lies at the origin.  Rotate it around the $z$ axis so that $P_3$, the plane containing $e_3$ and $f_3$ contains the $y$ axis.

If we may choose $P_2=P_3$ while preserving the above properties, we do so.  If not, and $P_3$ is steeper than $P_2$
rotate them around the $z$ axis so that the upper half plane of each $P_i$ (the portion above the $xy$-plane) has non-positive $x$-coordinates. 
If $P_2$ is steeper than $P_3$ rotate them so that the upper half plane of each $P_i$ has non-negative $x$-coordinates.
 $v_2$ and $v_3$ remain extremum for $K_2$ and $K_3$ respectively 
and each $P_i$ still contains the $y$-axis.

If two of the knots are planar, recall that we chose them to be $K_1$ and $K_3$.  Since $D_3$ is a flat disk totally contained in $P_3$, and neither $D_1$ nor $D_2$ have points with positive $z$ coordinates, we may rotate $K_3$, $D_3$, and $P_3$ around the $y$ axis until $P_3 = P_2$ without introducing any new intersections of the disks.  If the knots are not planar then we now apply Lemma~\ref{lemma:extremaldisks2} to scale the knots.

Now a rotation of $K_3$, $D_3$, and $P_3$ around the $y$-axis through the acute angle between $P_2$ and $P_3$ keeps $v_3$ at the origin and preserves the properties of  Corollary~\ref{cor:planar}.  We have set up the planes so that no matter whether the upper half plane of $P_3$ was above or below the upper half-plane of $P_2$, the rotation goes in the same direction.  It is a clockwise rotation around the $y$-axis if looking towards the origin from a point on the positive $y$-axis.  Equivalently a point of the form $(x,y,z)$ with $z>0$ will see its $x$ coordinate decrease as a result of the rotation.

As we rotate $P_3$ around the $y$-axis onto $P_2$, the total rotation is less than 90 degrees and any points other than $v_3$ where $D_3$ now intersects the $xy$-plane have negative $x$ values since they all had positive $z$ values before the rotation.  Thus they are disjoint from the planar portion of $D_1$, which had only non-negative $x$ values.  We have, of course, asserted that this is the only portion of $D_1$ which $D_3$ can intersect.  Similarly by Corollary~\ref{cor:planar} the only points of $D_3$ that could intersect $D_2$ are in the planar portion of $D_3$ and aside from $v_3$ this remains above the $xy$-plane.  Since $D_2$ is strictly below the $xy$-plane aside from $v_2$, we know that even after the rotation the only intersection of $D_2$ and $D_3$ must be at the origin.

Thus we may now assume that $P_2=P_3$ and we rename the new plane $P$.  
  Thus $D_i \cap D_j$ is the origin for $i \neq j$ and $e_2$, $f_2$, $e_3$ and $f_3$ all lie in the same plane $P$.  Note that $v_3$ may no longer be a global minimum with respect to $z$, but this is not a problem since we now completely understand the intersection patterns of the disks.

We do not want any of $e_2$, $f_2$, $e_3$ and $f_3$ to be collinear, but by general position, this may be ensured by rotating one of the knots by $\epsilon$ around the line perpendicular to $P$ and through the origin without creating any new intersections.

 Note that $J_s \cap D_t = \empty set$ for $s \neq t$.  Since each $D_s$ and $J_t$ is  a pair of disjoint compact sets we may find a minimum distance $d$ from $J_s$ to $D_t$ over all $s \neq t$.  
For the rest of the paper we will make sure that no $D_s$ is moved more than $d/2$ and thus this disjoint property will be preserved.  {\em Thus from here on out $D_s \cap D_t$ for $s \neq t$ will only occur in the flat triangular portions of the disks that lie in an $\epsilon$ neighborhood of the origin.}

\medskip

We have completed the only scaling we need in the proof of Theorem~\ref{thm:poly} and no scaling is needed in the proof of
Theorem~\ref{thm:polyplanar}.   Otherwise the proofs of the two theorems are identical.  Although $P$ may not be the $yz$-plane, we will picture it in this manner in the figures since that will not impact the future arguments (we only use the fact that $P$ exists and contains the $y$-axis, not the specific angle it makes with the $xy$-plane).

Let $f_2$ be the least steep edge from the collection $\{e_2, f_2, e_3, f_3 \}$ (the absolute value of the slope of $f_2$
in $P$  is less than the absolute value of the slopes of $e_2$, $e_3$ and $f_3$ in the same plane).  If we chose the knots wisely from the start this is a safe assumption, but if not this may require returning to the start of the argument and relabeling of $K_2$ and $K_3$ and going through the above steps, all of which will still work fine and lead to this desired steepness result.  

We now fix $K_1$ for the rest of the proof and move the other two knots slightly starting with $K_2$.

\begin{claim}
Given  disks $\{D_1, D_2, D_3\}$ intersecting only at the origin and bounded by knots $\{K_1, K_2, K_3\}$ as above, then for any $\epsilon >0$, let $\{J_1, J_2, J_3\}$ be equal to $\{D_1, D_2, D_3\}$ minus the portion of the $D_i$'s in an open $\epsilon$ ball around the origin, and given any translation of $\R^3$ acting on a given $D_i$ or any rotation of that $D_i$ about a fixed axis $l$, then there exists an $\epsilon'$ such that any translation of distance less than $\epsilon'$ leaves the $J_i$'s pairwise disjoint.
Similarly there exists an $\alpha > 0$ such that any rotation about $l$ of angle less than $\alpha$ leaves the $J_i$'s pairwise disjoint.
\label{claim:simpint}

\end{claim}

{\em Proof of Claim~\ref{claim:simpint}:} 
Since $J_i \cap J_j = \empty set$ and they are both compact, there is a positive minimal distance $s$ between any point in $J_i$ and $J_j$.  Setting $\epsilon' < s$  will ensure that translating one of the disks less than $\epsilon'$ cannot create an intersection.  Similarly after fixing an axis of rotation we can pick a small enough angle $\alpha$ such that rotating $J_i$ will move no point of $J_i$ more than $s/2$ completing the proof of the claim.

\medskip

We may pick any $\epsilon > 0$ and be certain that no point on any disk will be moved more than $\epsilon$ for the rest of the proof.  
The virtue of Claim~\ref{claim:simpint} is that we now know that during all our remaining manipulations of the disks and knots the existing components of intersection may change in size and shape (or even go away), but no new intersections will be introduced.  Also all intersections will remain in an $\epsilon$ ball neighborhood of the origin and on the flat triangular pieces of the $D_i$'s running between $e_i$ and $f_i$ (i.e. the complement to $J_i$ in $D_i$) .

Let $ {\vec w}$ be the vector with its head at $v_2$ and parallel to $f_2$ (its tail may be thought of as lying on the other vertex of $f_2$) as in Figure~\ref{fig:origin}.  Translate $K_2$ by adding $\epsilon * {\vec w}$ to every point on $K_2$ for a sufficiently small $\epsilon$ in order to translate $K_2$  (and $D_2$) minimally up in a direction parallel to $f_2$. We want to be certain that $e_2 \cap e_1$ remains nontrivial and that no new intersections are introduced outside of a neighborhood of the origin.
By Claim~\ref{claim:simpint} choosing a sufficiently small $\epsilon$ will ensure all of these properties, as would any positive translation
smaller than $\epsilon$.
Now
$v_2$ is very close to, but above $v_1$, $f_2$ intersects the origin ($v_1$), and $e_2$ intersects $e_1$ in some point other than $v_1$.  $K_1$ and $K_2$ look as they do in Figure~\ref{fig:polygonal} and we need to reposition $K_3$ to match the figure.

The fact that $f_2$ is not as steep as $e_3$ and $f_3$ ensures that  both $e_3$ and $f_3$ are on the same side of the line containing $f_2$ in $P$.  $f_2$ being less steep than $e_2$ and on its right in $P$ ensures that the points of $f_2$ all have non-negative $y$-value.
Let ${\vec v}$ be a vector with its head on $v_3$ at the origin and its tail between $e_3$ and $f_3$ as in Figure~\ref{fig:origin}.
Translate $K_3$
by $\delta * {\vec v}$ for a small $\delta$.  This
keeps $e_3$ and $f_3$ in $P$.
Choosing a sufficiently small $\delta$  again makes sure that all changes in the intersections of the disks occur in a neighborhood of the origin, that
$K_2 \cap K_3$ consists of exactly two points, $e_3 \cap f_2$ and $f_3 \cap f_2$.
Finally to complete the figure pick $l$, a line in $P$ parallel to $f_2$, but separating $f_2$ from $v_3$.  Let $l \cap e_3$ be called $p_e$ and let $l \cap f_3$ be called $p_f$.  We must pick $l$ close enough to $f_2$ so that $p_e$ is above $e_1$.  This is easy to do since  $e_1$ lies on the $y$-axis and we need only make sure that $p_e$ has positive
$z$-coordinate.  The point $e_3 \cap f_2$ has positive $z$ coordinate so if $l$ is sufficiently close to $f_2$ then $p_e$ will, too.  On the other hand, $f_2 \cap f_3$  has negative $z$-coordinate, and $p_f$ is below this point, $p_f$ will have negative $z$-coordinate.

Recall that we have already established  that in Figure~\ref{fig:polygonal} we may assume the $D_i$ are all disjoint from the neighborhood of the origin depicted except in the obvious flat triangular  sub-disks and that the $D_i$ are disjoint from each other outside of the figure.
%
%\begin{cor}
%We may choose each of the $D_i$ to be flat in a neighborhood of the origin and then $D_i \cap D_j$ (the disks bounded by $K_i$ and $K_j$) in Figure~\ref{fig:polygonal}
%will be contained in the flat triangular  sub-disks in the figure for each $i \neq j \in \{1,2,3\}$.  The intersection for each pair of disks is shown in Figure~\ref{fig:originint}
%\end{cor}
%

\begin{figure}
	\centering
	\includegraphics[width=0.45\textwidth]{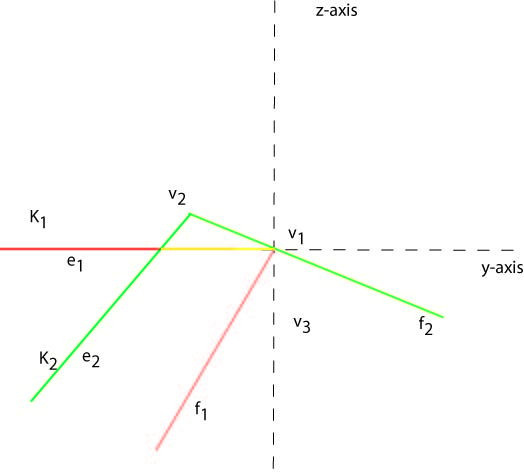}
	\includegraphics[width=0.45\textwidth]{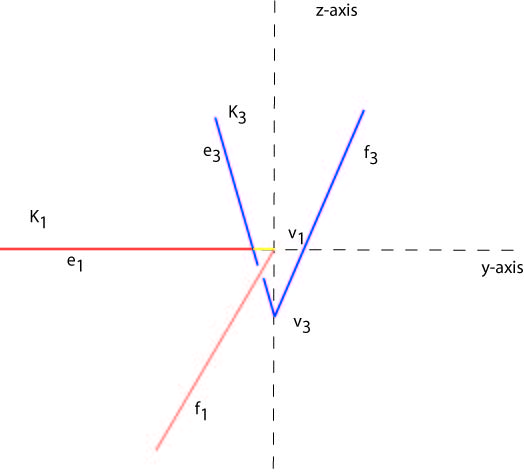}
	\includegraphics[width=0.45\textwidth]{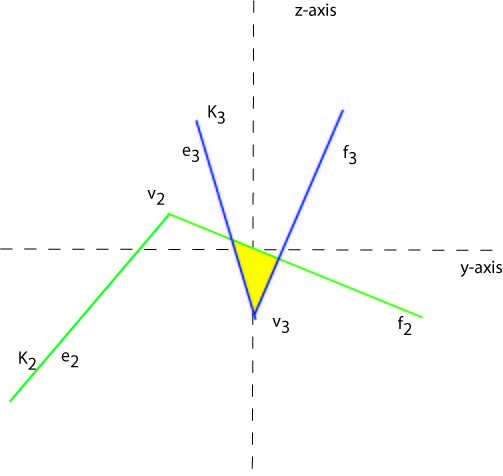}
	\caption{The initial intersections of the disks $D_1 \cap D_2$, $D_1 \cap D_3$, and $D_2 \cap D_3$
	respectively are shown in gold. They are not yet in general position.}
	\label{fig:brstick4}
	\label{fig:originint}

	\end{figure}

%	
%\begin{proof}
%Let $H_2$ be the horizontal plane containing $v_3$ and $H_3$  be the horizontal plane containing $v_2$.
%The only portion of $D_2$ contained above $v_3$ (and thus above $H_2$) is the triangular subdisk contained in the picture.  The only portion of $D_3$ contained below $v_2$ and (thus below $H_3$) is the triangular subdisk contained in the picture.
%If $K_1$ is planar, then we may assume that
%the disk bounded by $K_1$ is contained in the $xy$-plane.
%%there is a plane, $H_1$ intersecting $K_1$ in exactly two points, one on
%%$e_1$ and the other on $f_1$. We may further choose the plane such that $v_1$, $K_2$, and $K_3$ are all on the same side of $H_1$.
%If $K_1$ is not planar, our scaling assures us there is a sphere $H_1$ centered at $v_1$,
%intersecting $K_1$ in exactly two points, one on $e_1$ and the other on $f_2$ and containing $v_1$, $K_2$, $D_2$, $K_3$, and $D_3$
%all on one side of the sphere and the other critical points of $K_1$ on the other side (we pick a sphere centered at the origin and with radius slightly larger than the bigger of the two values $d_2$ and $d_3$ defined earlier).  In either case $H_1$  ensures that the portion of $D_1$ near the other two disks is contained in the $xy$-plane.  $D_2$ and $D_3$ are just flat triangular disks between $H_2$ and $H_3$ and the rest of $D_2$ is below all of $D_3$ (since it is below $H_2$) and the rest of $D_3$ is above all of $D_2$ since it is above $H_3$).  Thus pairwise the disks are disjoint outside of the pictured area near the origin.
%
%
%\end{proof}

Now we want to put the knots and disks in general position. This process will take us from Figure~\ref{fig:polygonal} to Figure~\ref{fig:brstick}.  Because general position is always easy to attain with infinitesimally small transformations we can, as mentioned earlier, pick a small number $\epsilon$ and no point of the knots or disks will move more than $\epsilon$ over the rest of the proof. This ensures that the only new intersection patterns between the disks will be the result of local changes in the current intersection patterns.

\begin{figure}
	\centering
	\includegraphics[width=0.65\textwidth]{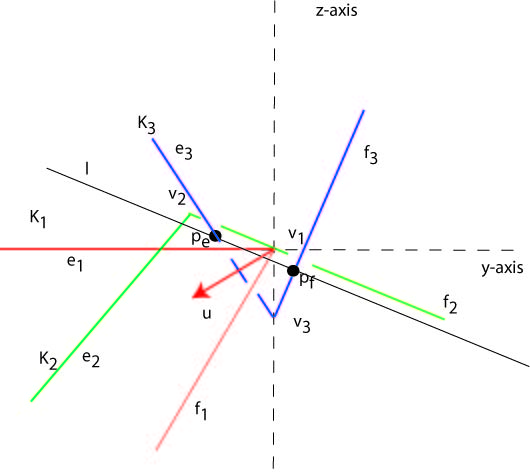}
	\caption{$K_3$ has been rotated and the Borromean rings will be formed once we translate $K_2$ slightly in the direction of horizontal vector  ${\vec u}$ which lies in the $xy$-plane with its tail at the origin and its head between $e_1$ and $f_1$.}
	\label{fig:rotated}

	\end{figure}
	
\begin{figure}
	\centering
	\includegraphics[width=0.65\textwidth]{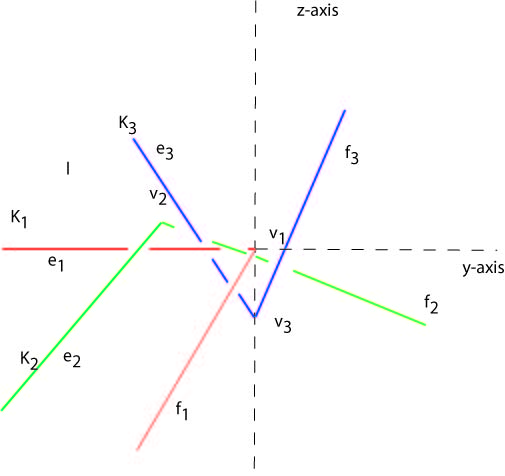}
	\caption{The Borromean Rings.}
	\label{fig:brstick}

	\end{figure}

Rotate $K_3$ around $l$ so that the $x$ coordinate of $v_3$ decreases and so that $D_3$ is in general position with respect to both $D_1$ and $D_2$ (although $D_1$ and $D_2$ are still not in general position with respect to each other).  Rotating by a small enough angle will ensure that no point on $K_3$ or $D_3$ moves more than $\epsilon$.
The rotation will fix $p_e$ and $p_f$, will cause all the points on the same side of $l$ as $v_3$ to have decreasing $x$ coordinates and all the  points of $e_3 \cup f_3$ on the other side to have increasing $x$ coordinates.
Before rotating, $D_3$ intersected $D_1$ in a single arc, a subset of $e_1$  running from $e_3 \cap e_1$ to $v_1$, including the single point
$l \cap D_1$.  After the rotation $D_1 \cap D_3$ will remain an arc, $l \cap D_1$ will be one endpoint, and the arc of intersection will rotate about this point.  The other end point will move away from the origin ($v_1$) to a point on $f_1$ with positive $x$ coordinate as in Figure~\ref{fig:brstick3}.

Before rotating $D_3$, $D_2 \cap D_3$ was a triangular subset of $P$ formed by intersecting the triangle subset of $D_2$ running from $e_2$ to $f_2$ and the analogous triangle on $D_3$ from $e_3$ to $f_3$.  The two triangles and thus the intersection contained the portion of $l$ running from $p_e$ to $p_f$.  After rotating this portion of $l$ will be the only portion of $D_3$ near the origin contained in $P$.  Since we have already established that all intersections will occur near the origin this means that $D_2 \cap D_3$ is exactly the arc of $l$ running from $p_e$ to $p_f$.   Now $D_3$ is in general position with respect to both $D_1$ and $D_2$.

Finally we must translate $K_2$ and $D_2$ slightly  so that $D_1 \cap D_2$ is in general position.  This will move $D_2 \cap D_3$ infinitesimally, but since these two disks are already in general position and the move will be minimal it will not be enough to change the intersection pattern of those two topologically so for our purposes we may think of it as essentially unchanged.  Before translating $D_2 \cap D_1$ is the subset of the edge $e_1$ running from $e_1 \cap e_2$ to $v_1 = e_1 \cap f_2$.
  Let ${\vec u}$ be a vector in the $xy$-plane with tail at the origin ($v_1$) and head on the triangular portion of $D_1$ between $e_1$ and $f_1$ as in Figure~\ref{fig:rotated}.  Translate $K_2$ by $\rho * {\vec u}$
where $|\rho * {\vec u}|$ is small enough to satisfy Claim~\ref{claim:simpint}.
Since $D_2$ and $D_3$ were in general position and our translation was minimal,
 $D_2 \cap D_3$ remains an arc as before (although it is no longer a subset of $l$).  $D_1$ and $D_2$ now are in general position and $D_1 \cap D_2$ becomes an arc from $e_2 \cap D_1$ to $f_2 \cap D_1$ that is parallel to, but now disjoint from $e_1$.

   All the disks are now in general position and the link looks locally like Figure~\ref{fig:brstick}.   The disks now each intersect the union of the other two in a cross as in Figure~\ref{fig:brstick5}.  It is not hard to show that this intersection pattern can only result from the Borromean Rings.  See, for example, \cite{H2}.  This concludes the proof of the two main theorems.

\bigskip

Although no set of three circles can be used to form the Borromean rings, Theorem~\ref{thm:polyplanar} has the following interesting Corollary.

\begin{cor}
Given any three circles, $C_1$, $C_2$, and $C_3$ and any $\epsilon >0$ there exists unknots
$K_1$, $K_2$, and $K_3$ with each $K_i$ contained in an $\epsilon$ tubular neighborhood of $C_i$  and isotopic to $C_i$ in that neighborhood, such that the Borromean Rings may be formed from $K_1 \cup K_2 \cup K_3$ through rigid motions of the components in $\R^3$.
\end{cor}

This follows immediately by picking planar polygonal unknots arbitrarily close to each $C_i$.  It then shows that while the Borromean Rings can't be formed out of three circles, they can in some sense come arbitrarily close to forming them.

%\begin{figure}
%	\centering
%	\includegraphics[width=0.45\textwidth]{sticks2c12.jpg}
%	\includegraphics[width=0.45\textwidth]{sticks2c13.jpg}
%	\includegraphics[width=0.45\textwidth]{sticks2c23.jpg}
%	\caption{The Intersections of the disks in the final format $D_1 \cap D_2$, $D_1 \cap D_3$, and $D_2 \cap D_3$
%	respectively are shown in gold.}
%	\label{fig:brstick3}
%
%	\end{figure}

\begin{figure}
	\centering
	\includegraphics[width=0.45\textwidth]{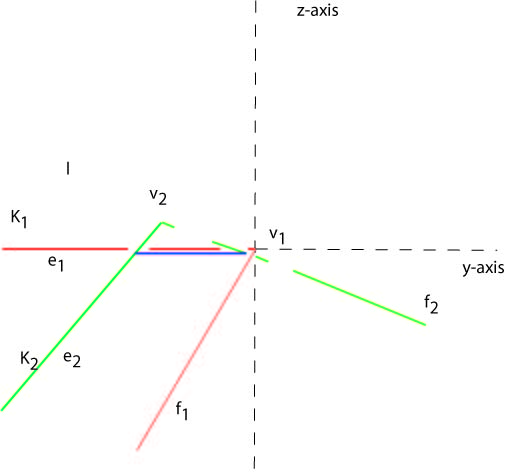}
	\includegraphics[width=0.45\textwidth]{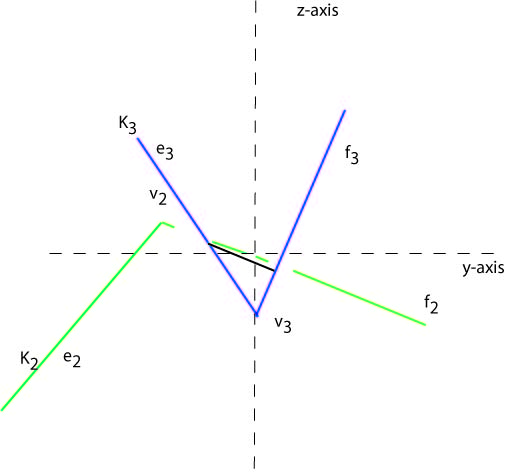}
\includegraphics[width=0.45\textwidth]{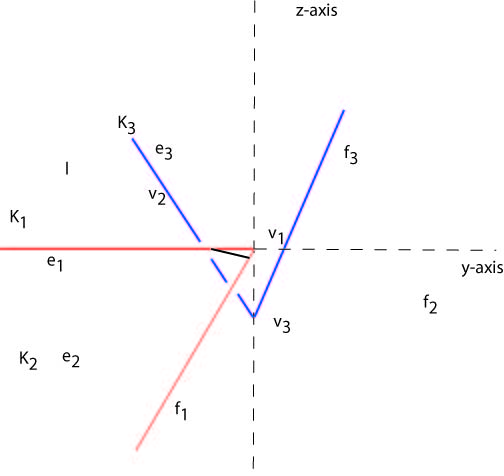}
	\caption{The Intersections of the disks in the final format $D_1 \cap D_2$, $D_1 \cap D_3$, and $D_2 \cap D_3$
	respectively are shown in black. Note that each time the black arc has end points on one of the knots (in the top left $K_1$, top right $K_2$, and bottom $K_3$) and on the interior of the disk bounded by the other knot.
 If all three arcs were drawn in the same picture we would see that the top two form a cross in the $xy$-plane intersecting in a single point on the interior of both arcs.  The third arc intersects the $xy$-plane
they would cross in a single point on its interior, which also is the unique point where it intersects the other two arcs. }
	\label{fig:brstick3}
	\end{figure}

\begin{figure}
	\centering
	\includegraphics[width=0.25\textwidth]{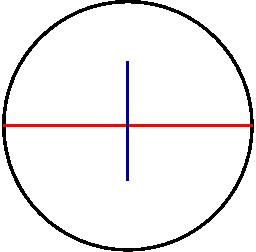}
	\caption{$D_r \cap (D_s \cup D_t)$ looks like the figure above
	for any distinct choices of $r,s,t \in \{1,2,3\}$.  This can only happen in the case of the
	Borromean rings.}
	\label{fig:brstick5}

	\end{figure}

\section{Conjectures, open questions, and a possible counter-example to Conjecture~\ref{conj:cook}}

%
%\begin{figure}
%	\centering
%	\includegraphics[width=0.65\textwidth]{stick4.jpg}
%	\caption{The Borromean Rings.}
%	\label{fig:onestick}
%
%	\end{figure}

In spite of our evidence in {\em partial} support of Cook's conjecture, we now state a conjecture of our own that would contradict it.

\begin{conjecture}
It is possible to find three unknotted curves, one of which is not a circle that cannot be used to form the Borromean rings. (Here scaling is not allowed.)
\end{conjecture}

Jason Cantarella suggests the following example as likely to satisfy this new conjecture (and arguably a counter-example to Cook's original conjecture).  Let $T$, shown in Figure~\ref{fig:torus}, be a torus that can be parameterized as follows:

$x= 5.1*\cos(\theta)+5*\cos(\psi)*\cos(\theta), y= 5.1*\sin(\theta)+5*\cos(\psi)*sin(\theta), z=5*\sin(\psi)], 0 \leq \psi \leq  2*\pi,  0 \leq \theta \leq  2*\pi$.

Let $L$ consist of two circles $K_1$ and $K_2$ of radius 10 together with an unknot $K_3$ that is isotopic {\em on $T$} to an $(n,1)$ torus knot, but consisting of $n$ arcs of meridian circles outside of a small neighborhood of the origin and then short arcs on the torus connecting adjacent arcs to complete a single knot.  An example of such a knot with $n=18$ is shown in Figure~\ref{fig:untorus}.  The knot consists of meridian curves like those in in Figure~\ref{fig:torus} away from a neighborhood of the origin.  Within a neighborhood of the origin there will be short, non-circular arcs.   These are in the dense central portion of the bottom picture in  Figure~\ref{fig:torus}.

Now because of the relatively large radii of $K_1$ and $K_2$ compared to the size of $T$, from their perspective the exposed portions of $K_3$ consist exclusively of arcs of circles.  The interactions in the construction of Theorem~\ref{thm:poly} are local and cannot work in this situation for the same reasons it would not work with three circles.  It is not as clear that the construction from Theorem~\ref{thm:polyplanar}, however, could not work, where scaling was allowed.  If we are allowed to scale the components, then if we scale $K_3$ up enough, then the non-circular portions become exposed to the other knots.  This, therefore, might be a counter-example to  Conjecture~\ref{conj:cook} if  scaling is not permitted, but fails to produce a counterexample if it is.   The exact phrasing of Cook's conjecture is ambiguous since the word ``arranged" could be interpreted to allow scaling or it could be interpreted not to, but it seems more likely to exclude scaling.    It is highly possible that this nuance could be the difference between the conjecture being true or false!

It is also worth noting that while from afar Figure~\ref{fig:untorus} appears to consist of arcs of circles as desired, due to the nature of computer generated images, if you truly look closely the knot pictured consists of straight segments that are very, very short.  No matter how short they are, this means that two computer generated ``circles," which also would consist of short straight segments, but would look like circles, together with the knot pictured would in truth be able to form the Borromean Rings without scaling by Theorem~\ref{thm:polyplanar}.  This shows how subtle the line is between knots that can form the Borromean Rings and those that cannot.

\begin{figure}
	\centering
	\includegraphics[width=0.75\textwidth]{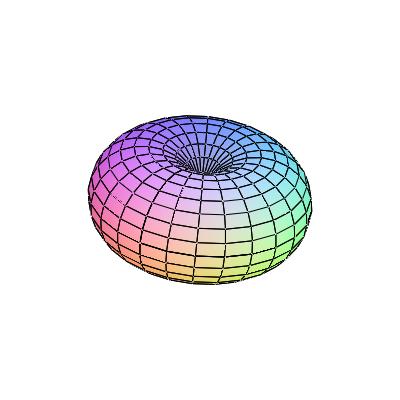}
	\caption{A torus.}
	\label{fig:torus}

	\end{figure}

\begin{figure}
	\centering
	\includegraphics[width=0.65\textwidth]{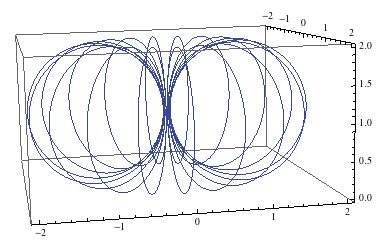}
	\includegraphics[width=0.65\textwidth]{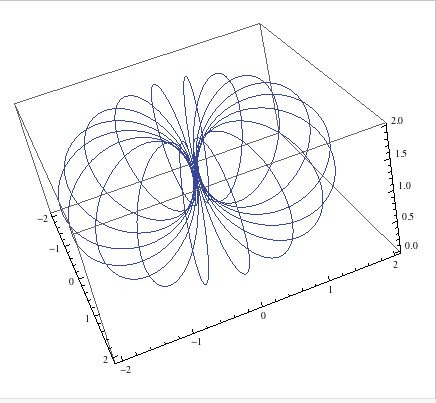}
	\includegraphics[width=0.65\textwidth]{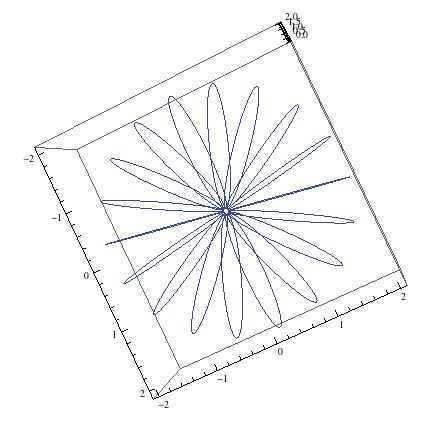}
	\caption{Various perspectives of an unknot embedded on a torus.}
	\label{fig:untorus}

	\end{figure}

We conclude with a few more open questions and conjectures.

\begin{conjecture}
Any three planar curves can be used to form the Borromean rings as long as at least one is not a circle.
\end{conjecture}

Planar was convenient and was necessary at times for the proofs in \cite{H3}, where it is shown that any three planar unknots (here they need not by polygonal) can always be used to form the Borromean rings through rigid transformations and scaling as long as one of them is not convex, but it is not clear that the theorem fails without it even if this proof does.

\begin{q}
Can any three unknots can be used to form the Borromean rings through rigid transformations and scaling applied to the individual components as long as at least one is not a circle?
\end{q}

Thanks to Jason Cantarella for the idea behind the link represented in Figure~\ref{fig:untorus} and to Matt Mastin for generating the images in Figure~\ref{fig:untorus}.

\bigskip\bigskip

%First put the title in the style you used before, for instance
\centerline{\textsc{References cited}}

\bigskip\bigskip

%%%%%%%%%%%%%%%%%%%%%%%%%%%%%%%
\catcode`\@=11
\def\@biblabel#1{\@ifnotempty{#1}{[\bfseries #1]}}

\renewenvironment{thebibliography}[1]{%
% \@bibtitlestyle
   \normalfont%\bibliofont
   \labelsep .5em\relax
   \renewcommand\theenumiv{\arabic{enumiv}}\let\p@enumiv\@empty
   \list{\@biblabel{\theenumiv}}{\settowidth\labelwidth{\@biblabel{#1}}%
     \leftmargin\labelwidth \advance\leftmargin\labelsep
     \usecounter{enumiv}}%
   \sloppy \clubpenalty\@M \widowpenalty\clubpenalty
   \sfcode`\.=\@m
}{%
   \def\@noitemerr{\@latex@warning{Empty `thebibliography'  
environment}}%
   \endlist
}

\newcommand{\bibsub}[1]{${}_{\mathbf{#1}}$}

\catcode`\@=12
%%%%%%%%%%%%%%%%%%%%%%%%%%%%%%

\end{document}